\documentclass[a4paper, 12pt]{amsart}
\usepackage{amsmath, amssymb, eucal, amscd, amstext}

\usepackage{enumerate}

\newtheorem{theorem}{Theorem}[section]
\newtheorem{corollary}[theorem]{Corollary}
\newtheorem{lemma}[theorem]{Lemma}

\theoremstyle{definition}
\newtheorem{definition}[theorem]{Definition}

\theoremstyle{remark}
\newtheorem{remark}[theorem]{Remark}

\newtheorem{notation}[theorem]{Notation}
\newtheorem{example}[theorem]{Example}
\numberwithin{equation}{section}

\newcommand{\I}{{\rm Int}}
\newcommand{\smnoind}{\smallskip\noindent}
\newcommand{\supp}{{\rm supp}}
\newcommand{\CL}{\mathcal{L}}

\newcommand{\FU}{\mathfrak{U}}
\newcommand{\FN}{\mathfrak{N}}

\newcommand{\CN}{\mathcal{N}}
\newcommand{\CU}{\mathcal{U}}
\newcommand{\bl}{\left}
\newcommand{\br}{\right}

\begin{document}

\baselineskip=17pt

\title[Linear orthogonality preservers of Hilbert bundles]{Linear orthogonality
preservers of Hilbert bundles}

\author{Chi-Wai Leung, Chi-Keung Ng  \and Ngai-Ching Wong }

\address[Chi-Wai Leung]{Department of Mathematics, The Chinese
University of Hong Kong, Hong Kong.}
\email{cwleung@math.cuhk.edu.hk}

\address[Chi-Keung Ng]{Chern Institute of Mathematics and LPMC, Nankai University, Tianjin 300071, China.}
\email{ckng@nankai.edu.cn}

\address[Ngai-Ching Wong]{Department of Applied Mathematics, National Sun Yat-sen University,  Kaohsiung, 80424, Taiwan, R.O.C.}
\email{wong@math.nsysu.edu.tw}

\thanks{The authors are supported by Hong Kong RGC Research Grant (2160255), National Natural Science Foundation of China (10771106), NCET-05-0219 and
Taiwan NSC grant (NSC96-2115-M-110-004-MY3).}

\keywords{automatic continuity, orthogonality preserving maps, module homomorphisms , local maps, Hilbert
C*-modules, Hilbert bundles}

\subjclass[2000]{46L08, 46M20, 46H40, 46E40}

\begin{abstract}
Due to the corresponding fact concerning Hilbert spaces, it is natural to ask if the linearity and the orthogonality structure
of a Hilbert $C^*$-module determine its $C^*$-algebra-valued inner product.
We verify this in the case when the $C^*$-algebra is commutative (or equivalently, we consider a Hilbert bundle over a
locally compact Hausdorff space).
More precisely, a $\mathbb{C}$-linear map $\theta$ (not assumed to be bounded)
between two Hilbert $C^*$-modules is said to be ``orthogonality
preserving'' if $\left<\theta(x),\theta(y)\right> =0$ whenever
$\left<x,y\right> =0$.
We prove that if $\theta$ is an orthogonality preserving map from a
full Hilbert $C_0(\Omega)$-module $E$ into another Hilbert $C_0(\Omega)$-module $F$ that satisfies
a weaker notion of $C_0(\Omega)$-linearity (known as ``localness''),
then $\theta$ is bounded and there exists $\phi\in C_b(\Omega)_+$ such that
$$
\left<\theta(x),\theta(y)\right>\ =\ \phi\cdot\left<x,y\right>, \quad \forall x,y \in E.
$$
On the other hand, if $F$ is a full Hilbert $C^*$-module over
another commutative $C^*$-algebra $C_0(\Delta)$, we show that a
``bi-orthogonality preserving'' bijective map $\theta$ with some
``local-type property'' will be bounded and satisfy
$$
\left<\theta(x),\theta(y)\right>\ =\ \phi\cdot\left<x,y\right>\circ\sigma, \quad \forall x,y \in E
$$
where $\phi\in C_b(\Omega)_+$ and $\sigma: \Delta \rightarrow \Omega$ is a homeomorphism.
\end{abstract}

\maketitle

\section{Introduction}

It is a common knowledge that the inner product of a Hilbert space
determines both the norm and the orthogonality structures; and
conversely, the norm structure determines the inner product structure.
It might be a bit less well-known that the orthogonality structure of a Hilbert
space also determines its norm structure.
Indeed, if $\theta$ is a linear
map between Hilbert spaces preserving orthogonality, then
it is easy to see that $\theta$ is a scalar multiple of an isometry (see
\cite{Blanco06,chmielinski05}).

We are interested in the corresponding relations for Hilbert $C^*$-modules.
Note that in the case of a commutative $C^*$-algebra $C_0(\Omega)$, Hilbert $C_0(\Omega)$-modules
are the same as
Hilbert bundles, or equivalently, continuous fields of Hilbert spaces over $\Omega$.
By modifying the proof of \cite[Theorem 6]{JW03} (see also \cite{Jerison50, Lau75, HsuWong}),
one can show that any surjective isometry between two continuous fields of Hilbert spaces with non-zero fibres
over each point is given by a homeomorphism and a field of unitaries.
Thus, the norm structure (and linearity) determines the unitary structure in this situation.

Our primary concern is the question of whether the orthogonality structure of
a Hilbert $C^*$-module determines its unitary structure.
More precisely, let $A$ be a $C^*$-algebra, and $E$ and $F$ be two Hilbert $A$-modules.
If $\theta: E\to F$ is an $A$-module homomorphism , which is not assumed to be bounded but
preserves orthogonality (that is, $\langle \theta(x), \theta(y) \rangle_A = 0$ whenever $\langle x, y \rangle_A = 0$),
we ask whether there is a central positive multiplier $u$ in $M(A)$ such that
$$
\left<\theta(e),\theta(f)\right>_A = u\left<e,f\right>_A, \quad \forall e,f \in E.
$$
When $A=\mathbb C$, it reduces to the case of Hilbert spaces.
Recently,  D. Ili\v{s}evi\'{c} and A. Turn\v{s}ek \cite{Turnsek-JMAA} gave a positive answer
in the case when $A$ is a standard $C^*$-algebra (that is,  $\mathcal{K}(H)\subseteq A \subseteq \CL(H)$).

In this article, we will give a positive answer when
$A$ is a commutative $C^*$-algebra (actually, we prove a slightly stronger result that
replaces the $A$-linearity with the localness property; see Definition \ref{def-local}).
On the other hand, we will also consider bijective bi-orthogonality preserving maps between Hilbert $C^*$-modules over
different commutative $C^*$-algebras.
We show that if such a map also satisfies certain local-type property (see Definition
\ref{def:quasi-local}) but not assumed to be bounded, then it is given by a homeomorphism
(between the base spaces) and a ``continuous field of unitaries''.
We remark that in this case of Hilbert C*-modules over different commutative C*-algebras, one cannot defines ``$A$-linearity'' but have to consider localness property.
This is one of the reasons for considering local maps.
We remark also that this case does not cover the case of Hilbert C*-modules over the
same commutative C*-algebra because we need to assume that the map is both bijective
and bi-orthogonality preserving.

Note that if $\Omega$ is a locally compact Hausdorff space and $H$ is a Hilbert space,
then $C_0(\Omega,H)$
is a Hilbert $C_0(\Omega)$-module.
As far as we know, even in this case our results are new, and the technique in
the proofs are non-standard and non-trivial comparing with those in the literatures \cite{AAK, AJ03, GJW03, jarosz:1990},
concerning separating or
zero-product preservers (although some statements look similar).
In a forthcoming paper of the authors, we will study the case when the underlying C*-algebra is not commutative.

\section{Terminologies and Notations}


Recall that a (right) \emph{Hilbert $C^*$-module} $E$ over a  $C^*$-algebra $A$
is a right $A$-module equipped with an $A$-valued inner product
$\langle \cdot, \cdot\rangle: E\times E\to A$ such that the following conditions hold
for all $x,y\in E$ and all $a\in A$.
\begin{enumerate}
    \item $\langle x,ya\rangle =  \langle x,y\rangle a$.
    \item $\langle x,y\rangle^* = \langle y,x\rangle$.
    \item $\langle x,x\rangle \geq 0$, and $\langle x,x\rangle =0$ exactly when $x=0$.
\end{enumerate}
Moreover,  $E$ is a Banach space equipped  with  the
norm $\|x\| = \|\langle x,x\rangle\|^{1/2}$.  We also call $E$ a \emph{Hilbert $A$-module} in this case.
A complex linear map $\theta: E\to F$ between two Hilbert $A$-modules is called an \emph{$A$-module homomorphism}
if $\theta(xa)=\theta(x)a$ for all $a\in A$ and $x\in E$.
See, for example, \cite{Lance95} or \cite{Rae98},
for a general introduction to the theory of Hilbert $C^*$-modules.
In this paper, we are interested in the case when the underlying $C^*$-algebra $A$ is abelian, that is,
$A=C_0(\Omega)$ consisting of all continuous complex-valued functions defined on a locally compact Hausdorff
space $\Omega$ vanishing at infinity.

\begin{definition}
\label{def-local} Let $A$ be a $C^*$-algebra.
Suppose that $E$ and $F$ are Hilbert $A$-modules.
A $\mathbb{C}$-linear map $\theta: E \rightarrow
F$ is said to be \emph{local} if $\theta(e)a = 0$ whenever $ea =0$ for
any $e \in E$ and $a\in A$.
\end{definition}

The idea of local linear maps is found in many researches in analysis. For
example, a theorem of Peetre  \cite{Pe60} states that local linear
maps of the space of smooth functions defined on a manifold modeled
on $\mathbb{R}^n$ are exactly linear differential operators (see \cite{Na85}).
This is further extended to the case of
vector-valued differentiable functions defined on a finite
dimensional manifold by Kantrowitz and Neumann \cite{KN} and
Araujo  \cite{Ar04}, and in the Banach $C^1[0,1]$-module setting by Alaminos et.\ al.\  \cite{Al08}.
Note that every $A$-module homomorphism  is local.
Conversely, every \emph{bounded} local map is an $A$-module homomorphism
(\cite[Proposition A.1]{LNW-auto-cont}).
See Remark \ref{rem:local-support-shringking}
below for more information.

\begin{notation}
Throughout this article, $\Omega$ and $\Delta$ are two locally
compact Hausdorff spaces, and $\Omega_\infty$ is the one-point
compactification of $\Omega$.
Moreover, $E$ and $F$ are
respectively, a (right) Hilbert $C_0(\Omega)$-module and a (right) Hilbert
$C_0(\Delta)$-module, while $\theta: E\rightarrow F$ is a
$\mathbb{C}$-linear map (not assumed to be bounded).
We denote by $\mathcal{B}_{C_0(\Omega)}(E,F)$ the set of all bounded $C_0(\Omega)$-module homomorphisms
 from $E$ into $F$.
For any
$\omega\in \Omega$, we let $\CN_\Omega(\omega)$ be the set
of all compact neighborhoods of $\omega$ in $\Omega$. If $S\subseteq
\Omega$, we denote by $\I_\Omega(S)$ the interior of $S$ in
$\Omega$. Moreover, if $U,V\subseteq \Omega$ such that the closure
of $V$ is a compact subset of ${\rm Int}_\Omega(U)$, we denote by
$\CU_\Omega(V,U)$ the collection of all $\lambda\in C_0(\Omega)$
with $0\leq \lambda \leq 1$, $\lambda \equiv 1$ on $V$ and $\lambda$
vanishes outside $U$.
\end{notation}

Note that any Hilbert $C_0(\Omega)$-module $E$ can be regarded as a Hilbert $C(\Omega_\infty)$-module, and
the results in \cite{DG} can be applied.
In particular, $E$ is the space of $C_0$-sections (that is,
continuous sections that vanish at infinity) of an (F)-Hilbert bundle
$\Xi^E$ over $\Omega_\infty$ (see \cite[p.~49]{DG}).

We define $|f|(\omega) :=
\|f(\omega)\|$ for all $f\in E$ and $\omega\in \Omega$.
For any closed subset $S\subseteq \Omega_\infty$ and $\omega\in \Omega_\infty$, we set
$$
K^E_S := \{f\in E: f(\omega) = 0, \forall \omega\in S\}
\quad {\rm and} \quad
I_\omega := \bigcup_{V\in \CN_{\Omega_\infty}(\omega)} K^E_V
$$
(for simplicity, we also denote $K^E_\omega := K^E_{\{\omega\}}$).
Notice that $K^E_\infty = E$ and the fibre of $\Xi^E$ at $\omega\in \Omega_\infty$ is given by
$\Xi^E_\omega = E/K^E_\omega$.
Furthermore, $K^E_S$ is a Hilbert $K^{C_0(\Omega)}_S$-module and
$K^E_S = \overline{E\cdot K^{C_0(\Omega)}_S}$.

On the other hand, we denote
$$
\Delta_\theta\ :=\ \left\{ \nu\in \Delta: \theta(E)\nsubseteq K_\nu^F\right\}
\ =\ \left\{ \nu\in \Delta: \theta(e)(\nu) \neq 0 \text{ for some } e\in E\right\}.
$$
Then $\Delta_\theta$ is an open subset of $\Delta$,
and we put
$$\Omega_E\ :=\ \left\{\omega\in \Omega: \Xi_\omega^E \neq (0)\right\}.$$

Let $\Omega_0\subseteq \Omega$ be an open set.
As in \cite[p.~10]{DG}, we denote by $\Xi^E|_{\Omega_0}$ the
restriction of $\Xi^E$ to $\Omega_0$ and by $E_{\Omega_0}$ the set of  $C_0$-sections on $\Xi^E|_{\Omega_0}$.
One can identify
$$C_0(\Omega_0)\ =\ K^{C_0(\Omega)}_{\Omega\setminus \Omega_0} \quad {\rm and} \quad E_{\Omega_0}\ =\ K^E_{\Omega\setminus \Omega_0}.$$

\section{Orthogonality preserving maps between Hilbert $C_0(\Omega)$-modules}\label{s:comm}

Let us first recall the following two technical lemmas from \cite[Lemmas 3.1 and 3.3, and Theorem 3.7]{LNW-auto-cont}
(see also \cite[Remark 3.4]{LNW-auto-cont}), which summarize, unify, and generalize techniques sporadically used
in the literatures \cite{AJ03, GJW03, jarosz:1990}.

\begin{lemma}
\label{lem:LNW3.1}
If $\sigma: \Delta_\theta \rightarrow
\Omega_\infty$ is a map satisfying $\theta\left(I_{\sigma(\nu)}^E\right)\subseteq K^F_\nu$ (for any $\nu\in \Delta_\theta$), then $\sigma$ is continuous.
\end{lemma}

\begin{lemma}
\label{lem:C_0-lin}
Let $\sigma: \Delta \rightarrow
\Omega$ be a map (not assumed to be continuous) such that
$\theta\left(I_{\sigma(\nu)}^E\right)\subseteq K^F_\nu$ for every $\nu\in \Delta$.

\smnoind (a) If $\FU_{\theta} := \bl\{\nu\in \Delta: \sup_{\|e\| \leq 1}
\|\theta(e)(\nu)\| = \infty\br\}$, then $\sigma(\FU_{\theta})$ is a finite set.

\smnoind (b) If $\FN_{\theta,\sigma} := \left\{\nu\in \Delta:
\theta\bl(K_{\sigma(\nu)}^E\br)\nsubseteq K_\nu^F\right\}$, then
$\FN_{\theta,\sigma}\subseteq \FU_{\theta}$ and $\sigma(\FN_{\theta,\sigma})$ consists of
non-isolated points in $\Omega$.

\smnoind
(c) If $\sigma$ is injective and sends isolated points in $\Delta$ to isolated points in $\Omega$,
then $\FN_{\theta,\sigma} = \emptyset$ and there exists a finite set $T$
consisting of isolated points of $\Delta$, a bounded linear map
$\theta_0:K^E_{\sigma(T)} \rightarrow K^F_{T}$ as well as linear maps
$\theta_\nu: \Xi^E_{\sigma(\nu)} \rightarrow \Xi^F_\nu$  for all $\nu\in T$,
such that $E = K^E_{\sigma(T)} \oplus \bigoplus_{\nu\in T} \Xi^E_{\sigma(\nu)}$,
$$
F = K^F_{T} \oplus \bigoplus_{\nu\in T} \Xi^F_\nu \quad {\rm and}
\quad \theta = \theta_0 \oplus \bigoplus_{\nu\in T}\theta_\nu.
$$
\end{lemma}

For any $\nu\in \Delta\setminus \FN_{\theta,\sigma}$, one can define $\theta_\nu: \Xi^E_{\sigma(\nu)} \rightarrow \Xi^F_\nu$ by
\begin{equation}
\label{def:theta_nu}
\theta_\nu\bl(e + K_{\sigma(\nu)}^E\br) = \theta(e) + K_\nu^F, \quad \forall e\in E,
\end{equation}
or equivalently, $\theta_\nu(e(\sigma(\nu))) = (\theta(e))(\nu)$  for all $e\in E$.

\begin{lemma}\label{lem:sp-B}
Let $\sigma$ and $\FU_{\theta}$ be the same as in Lemma \ref{lem:C_0-lin}.
Suppose, in addition, that $\sigma$ is injective and $\theta$ is orthogonality preserving.
Then there exists a bounded function $\psi: \Delta\setminus \FU_{\theta} \rightarrow \mathbb{R}_+$ such that
\begin{equation}
\label{theta(f)=}
\langle\theta(e), \theta(g)\rangle (\nu)\ =\ \psi(\nu)^2 \langle e, g \rangle (\sigma(\nu)), \quad \forall
e,g\in E, \forall \nu\in \Delta\setminus \FU_{\theta}.
\end{equation}
Moreover, for each $\nu\in \Delta_\theta$, there is an isometry $\iota_\nu:
\Xi^E_{\sigma(\nu)} \rightarrow \Xi^F_\nu$ such that
$$
\theta(e)(\nu)\ =\ \psi(\nu) \iota_\nu(e(\sigma(\nu))), \quad
\forall e\in E, \forall \nu\in \Delta_\theta\setminus \FU_{\theta}.
$$
\end{lemma}
\begin{proof}
Fix any $\nu\in \Delta_\theta\setminus \FU_{\theta}$. By Lemma
\ref{lem:C_0-lin}(b), the map $\theta_\nu$ as in
\eqref{def:theta_nu} is well-defined. Suppose that $\eta_1$ and
$\eta_2$ are orthogonal elements in $\Xi^E_{\sigma(\nu)}$ with
$\eta_1 \neq 0$ (it is possible because   $\Delta_\theta\setminus
\FN_{\theta,\sigma}\subseteq \sigma^{-1}(\Omega_E)$), and
$g_1,g_2\in E$ with $g_i(\sigma(\nu)) = \eta_i$ for $i=1,2$ . If $V\in
\CN_\Omega(\sigma(\nu))$ is such that $g_1$ is non-vanishing on $V$,
then by replacing $g_2$ with
$$
\bl(g_2 - \frac{\langle g_2, g_1 \rangle}{|g_1|^2} g_1\br)\lambda
$$
where $\lambda\in \CU_\Omega(\{\sigma(\nu)\}, V)$, we see that there
are orthogonal elements $e_1,e_2 \in E$ with $e_i(\sigma(\nu)) =
\eta_i$ for $i=1,2$ . Hence, $\theta_\nu$ is non-zero (because $\nu\in
\Delta_\theta$) and is an orthogonality preserving
$\mathbb{C}$-linear map between Hilbert spaces. Consequently, there
exist an isometry $\iota_\nu: \Xi^E_{\sigma(\nu)} \rightarrow
\Xi^F_\nu$ and a unique scalar $\psi(\nu) > 0$ such that $\theta_\nu
= \psi(\nu) \iota_\nu$. For any $\nu\in \Delta\setminus
\Delta_\theta$, we set $\psi(\nu) = 0$. Then clearly
\eqref{theta(f)=} holds. Next, we show that $\psi$ is a bounded
function on $\Delta\setminus \FU_{\theta}$. Suppose that it is not
the case. Then there exist distinct points $\nu_n \in \Delta_\theta
\setminus \FU_{\theta}$ such that $\psi(\nu_n) > n^3$. If $e_n \in
E$ with $\|e_n\| = 1$ and the modular function $|e_n|(\sigma(\nu_n))
=\sqrt{\left<e_n,e_n\right>}(\sigma(\nu_n)) \geq (n-1)/{n}$
(note that $\nu_n\in \sigma^{-1}(\Omega_E)$), then because of
\eqref{theta(f)=},
$$|\theta(e_n)|(\nu_n)\ =\ \psi(\nu_n)|e_n|(\sigma(\nu_n))\ >\ n^2(n-1).$$
As $\{\sigma(\nu_n)\}$ is a set of distinct points (note that
$\sigma$ is injective), by passing to a
subsequence if necessary, we can assume that there are $U_n\in
\CN_{\Omega}(\sigma(\nu_n))$ such that $U_n \cap U_m = \emptyset$
when $m\neq n$. Now, pick any $V_n\in \CN_\Omega(\sigma(\nu_{n}))$
with $V_n \subseteq {\rm Int}_{\Omega}(U_n)$ and choose a function
$\lambda_n\in \CU_\Omega(V_n, U_n)$ for all $n\in \mathbb{N}$. Define $e
:= \sum_{k=1}^\infty \frac{e_k\lambda_k^2}{k^2} \ \in\ E$. For any
$n\in \mathbb{N}$, as $n^2e - e_n\lambda_n^2\in\ K^E_{U_n}$ and $e_n
- e_n\lambda^2_n = e_n(1 - \lambda^2_n)\in K^E_{V_n}$, we have
$$
\left\|\theta (e)\right\|\ \geq\ \left\|
\theta(e)(\nu_n)\right\|\ = \ \frac{\left\|
\theta(e_n\lambda_n^2)(\nu_n)\right\|}{n^2} \ = \  \frac{\left\|
\theta(e_n)(\nu_n)\right\|}{n^2} \ >\ n-1
$$
(by the relation between $\theta$ and $\sigma$) which is a contradiction.
\end{proof}

\subsection{Hilbert bundles over the same base space}

\begin{remark}\label{rem:local-support-shringking}
For any $e\in E$, we denote
$$\supp_\Omega\ \! e\ :=\ \overline{\{\omega\in \Omega: e(\omega) \neq 0\}}.$$
It is not hard to check that the following statements are equivalent
(which tells us that local maps are the same as \emph{support shrinking maps} \cite{GJW03}):
\begin{enumerate}[(i)]
\item $\theta$ is local (see Definition \ref{def-local});
\item $\theta\bl(K_V^E\br) \subseteq K_V^F$ for any non-empty open set $V$;
\item $\supp_\Omega\ \! \theta(e) \subseteq \supp_\Omega\ \! e$ for every $e\in E$;
\item $\supp_\Omega\ \! \theta(e)\lambda \subseteq \supp_\Omega\ \! e$ for each $e\in E$ and $\lambda\in C_0(\Omega)$.
\end{enumerate}
\end{remark}

\begin{theorem}
\label{thm:op+ql} Let $\Omega$ be a locally compact Hausdorff space, and let $E$ and $F$
be two Hilbert $C_0(\Omega)$-modules. Suppose that $\theta: E \rightarrow
F$ is an orthogonality preserving local $\mathbb{C}$-linear map.
The following assertions hold.

\smnoind
(a) $\theta\in \mathcal{B}_{C_0(\Omega)}(E, F)$.

\smnoind
(b) There is a bounded non-negative
function $\varphi$ on $\Omega$ which is continuous on $\Omega_E$ such that
$$\langle \theta(e) , \theta(g) \rangle\ =\ \varphi \cdot \langle e, g \rangle,  \quad \forall e,g\in E.$$

\smnoind (c) There exist a strictly positive element
$\psi_0\in C_b(\Omega_\theta)_+$ and $J\in \mathcal{B}_{C_0(\Omega_\theta)}(E_{\Omega_\theta}, F_{\Omega_\theta})$ such that
the fiber map $J_\omega$
is an isometry for each $\omega\in \Omega_\theta$ and
$$
\theta(e)(\omega) = \psi_0(\omega) J(e)(\omega), \quad \forall e\in E, \forall \omega\in \Omega_\theta.
$$
\end{theorem}
\begin{proof}
Note that the conclusions of Lemmas \ref{lem:C_0-lin} and \ref{lem:sp-B} hold for $\Omega = \Delta$ and $\sigma = {\rm id}_\Omega$.

\smnoind
(a) By Remark \ref{rem:local-support-shringking} and Lemma \ref{lem:C_0-lin}(c), we see that $\theta$
is a $C_0(\Omega)$-module homomorphism.
Furthermore, as $\theta_\nu$ (as in Lemma \ref{lem:C_0-lin}(c)) is an orthogonality preserving (and hence bounded) linear map between Hilbert
spaces for any $\nu\in T$ (where $T$ is as in Lemma \ref{lem:C_0-lin}(c), with $\sigma = {\rm id}_\Omega$),
we know from Lemma \ref{lem:C_0-lin}(c) that $\theta$ is bounded (note that $T$ is finite).

\smnoind
(b) By part (a), $\FU_\theta = \emptyset$.
Thus, Lemma \ref{lem:sp-B} tells us that there exists a bounded non-negative function $\psi$ on $\Omega$ with
$\langle \theta(e), \theta(f) \rangle = |\psi|^2 \cdot \langle e, f \rangle$.
Let $\omega\in \Omega_E$ and pick any $e\in E$ such that there is $U_\omega\in \CN_\Omega(\omega)$
with $e(\nu)\neq 0$ for all $\nu\in U_\omega$.
Then $\psi(\omega) =
\frac{|\theta(e)|(\omega)}{|e|(\omega)}$  for all $\omega\in U_\omega$.
Hence $\psi$ is continuous at $\omega$, and $\varphi(\omega) = \psi(\omega)^2$ is the required function.

\smnoind
(c) Note that $\Omega_\theta \subseteq
\Omega_E$ because of part (a).
Since $\varphi(\omega) > 0$ ($\omega\in \Omega_\theta$), we
know from part (b) that $\psi = \varphi^{1/2}$ gives a strictly positive element $\psi_0$
in $C_b(\Omega_\theta)_+$.
The equivalence in \cite[(2.2)]{DG} (consider $E$ and $F$ as Hilbert
 $C(\Omega_\infty)$-bundles) tells us that the restriction of $\theta$
 induces a bounded Banach bundle map, again denoted by $\theta$, from
$\Xi^E|_{\Omega_\theta}$ into $\Xi^F|_{\Omega_\theta}$.
For each
$\eta\in \Xi^E|_{\Omega_\theta}$, we define $J(\eta) :=
\psi_0(\pi(\eta))^{-1}\theta(\eta)$ (where $\pi: \Xi^E \rightarrow
\Omega$ is the canonical projection). Then $J:\Xi^E|_{\Omega_\theta}
\rightarrow\Xi^F|_{\Omega_\theta}$ is a Banach bundle map (as
$\eta\mapsto \psi_0(\pi(\eta))^{-1}$ is continuous) which is an
isometry on each fibre (hence $J$ is bounded) such that
$\theta(\eta) = \psi(\pi(\eta))J(\eta)$.
This map $J$ induces a map, again denoted by $J$, in $\mathcal{B}_{C_0(\Omega_\theta)}(E_{\Omega_\theta}, F_{\Omega_\theta})$ that satisfies the requirement of part (c).
\end{proof}

It is natural to ask if one can find $\varphi\in C_b(\Omega)$ such that the conclusion of Theorem \ref{thm:op+ql}(b) holds.
Unfortunately, the following example tells us that it is not the case in general.

\begin{example}
Let $\Omega = \mathbb{R}_\infty$, the one-point compactification of the real line $\mathbb R$.
Consider $E = C_0(\mathbb{R}) = F$ as Hilbert $C(\Omega)$-modules and $\theta(f)(t) = f(t) \cos t$
for all $f\in E$ and $t\in \mathbb{R}$.
Then $\Omega \setminus \Omega_E = \{\infty\}$ and $\varphi(t) = \cos t$ for any $t\in \mathbb{R} = \Omega_E$.
Thus, one cannot extend $\varphi$ to a continuous function on $\Omega$.
\end{example}

Now, we can obtain the following commutative analogue of \cite[2.3]{Turnsek-JMAA}.
This, together with Corollary \ref{full+surj}, asserts that the orthogonality structure of a Hilbert bundle determines essentially its unitary structure, as we claimed in the Introduction.
Note also that a large portion of Lemma \ref{lem:C_0-lin} were used to deal with the possibility of  $\theta(K_{\sigma(\nu)}^E)
\nsubseteq K_\nu^F$ (such situation does not exist for $C_0(\Omega)$-module homomorphism), and this corollary actually has a much easier proof.

\begin{corollary}
\label{cor:comm-Turnsek}
Let $\Omega$ be a locally compact Hausdorff space, and $E$ and $F$ be two
Hilbert $C_0(\Omega)$-modules. Suppose that $\theta : E\rightarrow F$
is a $C_0(\Omega)$-module homomorphism  which preserves orthogonality. Then $\theta$
is bounded and there exists a bounded non-negative function $\varphi$ on $\Omega$ that is
continuous on $\Omega_E$ and satisfies $\langle \theta(e), \theta(f) \rangle
 = \varphi \cdot \langle e, f \rangle$ for all $e,f\in E$.
\end{corollary}

Recall that a Hilbert $C_0(\Omega)$-module $E$ is \emph{full} if the
$\mathbb{C}$-linear span, $\langle E, E \rangle$, of
$$
\{\langle e, f \rangle: e,f\in E\}
$$
is dense in $C_0(\Omega)$.

\begin{remark}
\label{full>T=0}
(a) $E$ is full if and only if $E \nsubseteq K^E_\omega$ for any $\omega\in \Omega$ (or equivalently, $\Omega_E = \Omega$).
In fact, if $E \subseteq K^E_\omega$, then $f(\omega) = 0$ for any $f\in \langle E, E \rangle$ and $E$ is not full.
Conversely, if $E$ is not full, then there exists $\omega\in \Omega$ such that $f(\omega) = 0$ for any $f\in \langle E, E \rangle$ (because the closure of $\langle E, E \rangle$ is an ideal of $C_0(\Omega)$) and $E\subseteq K^E_\omega$.

\smnoind
(b) If $E$ is full, then by part (a), the function $\varphi$ in Theorem \ref{thm:op+ql}(b) (and Corollary \ref{cor:comm-Turnsek}) is an element of $C_b(\Omega)$.
However, there is no guarantee that this function is strictly positive.

\smnoind
(c) Suppose that $F$ is full and $\theta$ is surjective orthogonality preserving local $\mathbb{C}$-linear map.
If there exists $\omega\in \Omega\setminus \Omega_\theta$, then $F = \theta(E) \subseteq K_\omega^F$ which contradicts the fullness of $F$ (see part (a)).
Consequently, $\Omega_\theta = \Omega$.
As $\theta\in \mathcal{B}_{C_0(\Omega)}(E,F)$ (by Theorem \ref{thm:op+ql}(a)), we see that $\Omega = \Omega_\theta \subseteq \Omega_E$ and $E$ is full (because of part (a)).
\end{remark}

\begin{corollary}
\label{full+surj}
Let $\Omega$ be a locally compact Hausdorff space, and let $E$ and $F$ be
two Hilbert $C_0(\Omega)$-modules. Suppose that $F$ is full and $\theta:
E \rightarrow F$ is an orthogonality preserving surjective local $\mathbb{C}$-linear
map. Then $\theta\in
\mathcal{B}_{C_0(\Omega)}(E,F)$. Moreover, there exist a strictly positive element $\psi\in C_b(\Omega)_+$
and a unitary $U\in \mathcal{B}_{C_0(\Omega)}(E, F)$ such that $\theta =
\psi \cdot U$.
\end{corollary}
\begin{proof}
Remark \ref{full>T=0}(c) tells us that $\Omega_\theta =
\Omega$.
By the surjectivity of $\theta$, the bounded Banach bundle map
$J$ in Theorem \ref{thm:op+ql} is a unitary on each fibre.
Therefore, the element $U\in\mathcal{B}_{C_0(\Omega)}(E,F)$ corresponding
to $J$ as given in \cite[(2.2)]{DG} is a unitary.
\end{proof}

\subsection{Hilbert bundles over different base spaces}

\begin{definition}
\label{def:quasi-local}
$\theta$ is said to be \emph{quasi-local} if it is bijective and for any $e\in E$ and $\lambda\in C_0(\Delta)$, we have
\begin{align}\label{cond:qausi-local}
\supp_\Omega\ \! \theta^{-1}(\theta(e)\lambda)\ \subseteq\ \supp_\Omega\ \! e.
\end{align}
\end{definition}

Note that if $\Delta = \Omega$, and if $\theta$ is both local and bijective (hence $\theta^{-1}$ is also local), then $\theta$ is quasi-local by Remark
\ref{rem:local-support-shringking}.

\begin{lemma}
\label{lem:bop+ql>qsep}
Suppose that $\theta$ is bijective and quasi-local and that both $\theta$ and $\theta^{-1}$ are orthogonality preserving.
Then $|\theta(e)| |\theta(g)| = 0$ whenever $e,g\in E$ with $\supp_\Omega\ \! e \cap \supp_\Omega\ \! g = \emptyset$.
\end{lemma}
\begin{proof}
Suppose on the contrary that there exist $e_1, e_2\in E$ and $\nu\in
\Delta$ such that $\supp_\Omega\ \! e_1 \cap \supp_\Omega\ \! e_2 =
\emptyset$ but $\|\theta(e_1)(\nu)\| \|\theta(e_2)(\nu)\| \neq 0$. As
$\theta$ is orthogonality preserving, we may assume that
$\theta(e_1)(\nu)$ and $\theta(e_2)(\nu)$ are two orthogonal unit
vectors in $\Xi_\nu^F$. Let $U,W\in \CN_\Delta(\nu)$ with
$W\subseteq \I_{\Delta}(U)$ and $\|\theta(e_i)(\mu)\| > 1/2$
for any $\mu\in U$.
Pick any $\lambda\in \CU_{\Delta}(W;U)$. Define
$h_i\in F\setminus\{0\}$ for $i=1,2$ by
$$
h_i(\mu)\ :=\ \begin{cases}
\theta(e_i)(\mu)\frac{\lambda(\mu)}{|\theta(e_i)|(\mu)} & \mu\in \I_{\Delta}(U)\\
0 & \mu\notin \I_{\Delta}(U)
\end{cases}
$$
and set $e'_i := \theta^{-1}(h_i)$.
The orthogonality of $h_1$ and $h_2$ (note that $e_1$ and $e_2$ are orthogonal), together with that of $h_1+h_2$
and $h_1-h_2$ (as $|h_1| = \lambda = |h_2|$), ensures the orthogonality of $e'_1$ and $e'_2$, as well as that of $e'_1+e'_2$ and $e'_1-e'_2$.
It follows that $|e'_1|=|e'_2|\neq 0$ which contradicts the fact  $|e_1'||e_2'| = 0$ (as $\theta$ is quasi-local).
\end{proof}

\begin{theorem}\label{thm:biortho-preserving}
Let $\Omega$ and $\Delta$ be locally compact Hausdorff spaces.
Suppose that $E$ is a full Hilbert $C_0(\Omega)$-module and $F$ is a
full Hilbert $C_0(\Delta)$-module. If $\theta : E\rightarrow F$ is a
bijective $\mathbb{C}$-linear map such that both $\theta$ and
$\theta^{-1}$ are quasi-local and orthogonality preserving, then
$\theta$ is bounded and
\begin{equation}
\label{std-form}
\theta(e)(\nu)\ =\ \psi(\nu)J_\nu(e(\sigma(\nu))), \quad \forall e\in E, \forall \nu\in \Delta,
\end{equation}
where $\sigma:\Delta \rightarrow \Omega$ is a homeomorphism, $\psi$ is a strictly positive element in $C_b(\Delta)_+$,
and $J_\nu$ is a unitary operator from $\Xi^E_{\sigma(\nu)}$ onto $\Xi^F_\nu$ such that for each fixed $f\in E$,
the map $\nu \mapsto J_\nu(f(\sigma(\nu)))$ is continuous.
\end{theorem}
\begin{proof}
We consider $E$ as a Hilbert $C(\Omega_\infty)$-module.
For each $\nu\in \Delta$, let
\begin{align*}
S_\nu\ :=\ \left\{\omega\in \Omega_\infty: \theta\bl(K_{\Omega_\infty\setminus W}^E\br) \nsubseteq K_\nu^F \text{ for every }
W\in \CN_{\Omega_\infty}(\omega)\right\}.
\end{align*}
We first show that $S_\nu$ is a singleton set.
Indeed, assume that $S_\nu = \emptyset$.
Then for any $\omega\in \Omega_\infty$, there is $W_\omega\in
\CN_{\Omega_\infty}(\omega)$ such that $\theta(K_{\Omega_\infty\setminus
W_\omega}^E)\subseteq K_\nu^F$.
Consider $\omega_1,...,\omega_n\in \Omega_\infty$ with
$$
\bigcup_{k=1}^n \I_{\Omega_\infty}(W_{\omega_k})\ =\ \Omega_\infty,
$$
and consider $\{\varphi_k\}_{k=1}^n$ to be a partition of unity
subordinate to $\{\I_{\Omega_\infty}(W_{\omega_k})\}_{k=1}^n$. Then
for any $e\in E$, we have $e\varphi_k \in K_{\Omega_\infty\setminus
W_{\omega_k}}^E$ and so $\theta(e)\in K_\nu^F$. This shows that $F =
K_\nu^F$ (as $\theta$ is surjective) which contradicts the fullness of $F$ (see Remark
\ref{full>T=0}(a)).
Now, assume that there are distinct elements
$\omega_1, \omega_2\in S_\nu$. Let $V_1\in
\CN_{\Omega_\infty}(\omega_1)$ and $V_2\in
\CN_{\Omega_\infty}(\omega_2)$ with $V_1\cap V_2 =
\emptyset$. By the definition of $S_\nu$, there exist $e_1,e_2\in
E$ with $\supp_\Omega\ \! e_i \subseteq\ V_i\setminus \{\infty\}$
and $\theta(e_i)(\nu)\neq 0$ for $i=1,2$  which contradict Lemma \ref{lem:bop+ql>qsep}. Thus, there
is a unique element $\sigma(\nu)\in \Omega_\infty$ with $S_\nu =
\{\sigma(\nu)\}$.
Next, we claim that
\begin{align}\label{theta(I)}
\theta\bl(I^E_{\sigma(\nu)}\br)\subseteq I^F_\nu, \quad \forall\nu\in \Delta.
\end{align}
Consider any $V\in \CN_{\Omega_\infty}(\sigma(\nu))$ and $e\in K_V^E$.
Pick any $U\in \CN_{\Omega_\infty}(\sigma(\nu))$
with $U \subseteq {\rm Int}_{\Omega_\infty}(V)$. By the definition
of $\sigma$, there exists $g\in K^E_{\Omega_\infty\setminus U}$ such
that $\theta(g)(\nu)\neq 0$.
Hence, there is $W\in
\CN_\Delta(\nu)$ such that $\theta(g)(\mu) \neq 0$  for all $\mu\in W$ and
Lemma \ref{lem:bop+ql>qsep} implies that $\theta(e)\in K_W^F$ as
claimed.
If there exists $\nu \in \Delta\setminus \Delta_\theta$, then for any $f\in F$, we
have $f(\nu) = 0$ (because $\theta$ is surjective) which
contradicts the fullness of $F$.
Thus, $\Delta_\theta = \Delta$ and $\sigma: \Delta\rightarrow \Omega_\infty$
is continuous (by Lemma \ref{lem:LNW3.1}).
As $\theta^{-1}$ is also quasi-local and orthogonality preserving,
a similar argument as the above gives a continuous map
$\tau:\Omega \rightarrow \Delta_\infty$ satisfying
$\theta^{-1}\bl(I^F_{\tau(\omega)}\br) \subseteq I^E_\omega$ for all $\omega\in \Omega$.
Now, the
argument of \cite[Theorem 5.3]{LNW-auto-cont} tells us that $\sigma$
is a homeomorphism from $\Delta$ to $\Omega$ such that
$$
\theta(e\cdot \varphi)\ =\ \theta(e) \cdot \varphi\circ \sigma, \quad \forall e\in E, \forall
\varphi\in C_0(\Omega),
$$
and by Lemma \ref{lem:C_0-lin}(c), there exists a finite set $T$ consisting of isolated points of
$\Delta$ such that $\theta$ restricts to a bounded map from
$K^E_{\sigma(T)}$ to $K^F_{T}$. Since
any $\nu\in T$ is an isolated point, $\theta$ induces an
orthogonality preserving (hence bounded) map $\theta_\nu$ from the Hilbert space
$\Xi^E_{\sigma(\nu)}$ onto the Hilbert space $\Xi^F_\nu$. This shows
that $\theta$ is bounded (because of Lemma \ref{lem:C_0-lin}(c) and the fact that $T$ is finite).
By Lemma \ref{lem:sp-B}, there is a surjective isometry $J_\nu:
\Xi^E_{\sigma(\nu)} \rightarrow \Xi^F_\nu$ such that
$$
\theta(e)(\nu)\ =\ \psi(\nu) J_\nu(e(\sigma(\nu))), \quad \forall e\in E, \forall\nu\in \Delta.
$$
Now the fullness of $E$ implies that $\psi(\nu) > 0$ (for every $\nu\in
\Delta$) and clearly $\nu\mapsto \frac{\theta(e)(\nu)}{\psi(\nu)}$
is continuous.
\end{proof}

Note that the assumption of $\theta^{-1}$ being orthogonality preserving is necessary in Theorem \ref{thm:biortho-preserving} as can be seen from the following example.

\begin{example}\label{eg:bijective-op}
Let $\Omega$ be a (non-empty) locally compact Hausdorff space, and $\Omega_2$ be the topological disjoint sum of two copies of $\Omega$ with
$j_1, j_2:\Omega \rightarrow \Omega_2$ being respectively the embeddings into
the first and the second copies of $\Omega$ in $\Omega_2$. Let $H$ be a
(non-zero) Hilbert space, and let $H_2$ be the Hilbert
space direct sum of two copies of $H$. Then the map $\theta:
C_0(\Omega_2, H)\longrightarrow C_0(\Omega, H_2)$ defined by
$$
\theta(f)(\omega)\ =\ (f(j_1(\omega)), f(j_2(\omega)))
$$
is a bijective $\mathbb{C}$-linear map preserving orthogonality satisfying Condition \eqref{cond:qausi-local}.
However, $\theta$ is not of the expected form.  Note that $\theta^{-1}$ does not preserve orthogonality.
\end{example}

\end{document}